\documentclass[11pt,a4paper,reqno]{amsart}
\usepackage[english]{babel}
\usepackage[T1]{fontenc}
\usepackage{verbatim}
\usepackage{palatino}
\usepackage{amsmath}
\usepackage{mathabx}
\usepackage{amssymb}
\usepackage{amsthm}
\usepackage{amsfonts}
\usepackage{graphicx}
\usepackage{esint}
\usepackage{xcolor}
\usepackage{mathtools}
\usepackage{overpic}
\usepackage{enumerate}

\usepackage[colorlinks = true, linkcolor={red!80!black}, citecolor = {blue!60!black}]{hyperref}
\author{Katrin F\"assler and Tuomas Orponen}
\title[Horizontal and $SL(2)$ lines]{A note on Kakeya sets of horizontal and $SL(2)$ lines}
\address{Department of Mathematics and Statistics\\ University of Jyv\"askyl\"a,
P.O. Box 35 (MaD)\\
FI-40014 University of Jyv\"askyl\"a\\
Finland}  \email{katrin.s.fassler@jyu.fi}
\email{tuomas.t.orponen@jyu.fi}

\date{\today}
\thanks{K.F.\ is supported by the Academy of Finland via the project
\emph{Singular integrals, harmonic functions, and boundary
regularity in Heisenberg groups}, grant No.\ 321696. T.O.\ is
 supported by the Academy of Finland via the project \emph{Incidences on Fractals}, grant No.\ 321896. }

\newcommand{\R}{\mathbb{R}}
\newcommand{\W}{\mathbb{W}}
\newcommand{\He}{\mathbb{H}}

\newcommand{\Hd}{\dim_{\mathrm{H}}}

\newcommand{\spa}{\operatorname{span}}

\def\Barint_#1{\mathchoice
          {\mathop{\vrule width 6pt height 3 pt depth -2.5pt
                  \kern -8pt \intop}\nolimits_{#1}}%
          {\mathop{\vrule width 5pt height 3 pt depth -2.6pt
                  \kern -6pt \intop}\nolimits_{#1}}%
          {\mathop{\vrule width 5pt height 3 pt depth -2.6pt
                  \kern -6pt \intop}\nolimits_{#1}}%
          {\mathop{\vrule width 5pt height 3 pt depth -2.6pt
                  \kern -6pt \intop}\nolimits_{#1}}}

\numberwithin{equation}{section}

\theoremstyle{plain}
\newtheorem{thm}[equation]{Theorem}
\newtheorem*{"thm"}{"Theorem"}

\newtheorem{cor}[equation]{Corollary}
\newtheorem{proposition}[equation]{Proposition}

\theoremstyle{definition}

\newtheorem{definition}[equation]{Definition}

\theoremstyle{remark}
\newtheorem{remark}[equation]{Remark}

\newcommand{\nref}[1]{(\hyperref[#1]{#1})}

\DeclareMathSymbol{\intop}  {\mathop}{mathx}{"B3}

\begin{document}

\begin{abstract} We consider unions of $SL(2)$ lines in $\R^{3}$. These are lines of the form 
$$L = (a,b,0) + \mathrm{span}(c,d,1),$$
where $ad - bc = 1$. We show that if $\mathcal{L}$ is a Kakeya set of $SL(2)$ lines, then the union $\cup \mathcal{L}$ has Hausdorff dimension $3$. This answers a question of Wang and Zahl.

The $SL(2)$ lines can be identified with \emph{horizontal lines} in the first Heisenberg group, and we obtain the main result as a corollary of a more general statement concerning unions of horizontal lines. This statement is established via a point-line duality principle between horizontal and \emph{conical} lines in $\R^{3}$, combined with recent work on \emph{restricted families of projections to planes}, due to Gan, Guo, Guth, Harris, Maldague, and Wang. 

Our result also has a corollary for Nikodym sets associated with horizontal lines, which answers a special case of a question of Kim.
\end{abstract}

\maketitle

\tableofcontents

\section{Introduction}

The purpose of this note is to study the Hausdorff dimension of unions of \emph{$SL(2)$ lines} in $\R^{3}$. Here is the definition of $SL(2)$ lines, following \cite{WaZh}:
\begin{definition}[$\mathcal{L}_{SL(2)}$]\label{def:SL} The family $\mathcal{L}_{SL(2)}$ consists of the following lines $L \subset \R^{3}$. Either $L$ is a line contained in the $xy$-plane, and $0 \in L$, or then
\begin{displaymath} L := L_{\alpha,\beta,\gamma,\delta} := (\alpha,\beta,0) + \spa(\gamma,\delta,1), \end{displaymath}
where $\alpha \delta - \beta \gamma = 1$.  \end{definition}
 
We also use the following notation. If $\mathcal{L}$ is any family of lines in $\R^{3}$, we write $\mathrm{dir}(\mathcal{L}) := \{e \in S^{2} : \ell \| \spa(e) \text{ for some } \ell \in \mathcal{L}\}$. Here is the main result of the note:
\begin{thm}\label{mainSL} Let $\mathcal{L} \subset \mathcal{L}_{SL(2)}$ be a set with $\mathcal{H}^{2}(\mathrm{dir}(\mathcal{L})) > 0$. Then 
\begin{displaymath} \Hd (\cup \mathcal{L}) = 3. \end{displaymath} \end{thm}

\begin{remark} Theorem \ref{mainSL} answers a question posed by Wang and Zahl in \cite[Section 1.2]{WaZh}. This question was motivated by earlier work of Katz and Zahl \cite{MR3868003}. Theorem \ref{mainSL} continues to hold if the full lines in $\mathcal{L}$ are replaced by line segments of positive length. We will discuss this briefly below \eqref{form41}. 

We have been informed that Katz, Wu, and Zahl have also proved Theorem \ref{mainSL} independently, using a different method. \end{remark}

The $SL(2)$ lines are essentially (up to a change in coordinates) the same as \emph{horizontal lines in the first Heisenberg group $\He = (\R^{3},\ast)$}, viewed as subsets of $\R^{3}$ (see Proposition \ref{prop1}). We will infer Theorem \ref{mainSL} from a more general statement concerning unions of these horizontal lines, Theorem \ref{main} below. We first need to define the concepts properly.

The family of all horizontal lines is denoted $\mathcal{L}(\He)$. The "Heisenberg" definition of these lines is the following. Let $\Pi_{0} := \{(x,y,0) : x,y \in \R\}$ be the $xy$-plane, and for $p \in \R^{3}$, let $\Pi_{p} := p \ast H_{0}$ be the \emph{left translate} of $\Pi_{0}$ by the Heisenberg group product
\begin{displaymath} (x,y,t) \ast (x',y',t') = \left( x + x',y + y',t + t' + \tfrac{1}{2}(xy' - x'y) \right). \end{displaymath}
Then, $\mathcal{L}(\He)$ consists of all the lines in $\Pi_{p}$ (for every $p \in \R^{3}$) which contain the point $p$.

The family $\mathcal{L}(\He)$ is a $3$-dimensional submanifold of the full ($4$-dimensional) family of lines in $\R^{3}$. In fact, the definition above of horizontal lines will not be used in the note: rather, we focus attention on the following parametrised subset of $\mathcal{L}(\He)$:
\begin{displaymath} \mathcal{L}'(\He) = \{\ell_{(a,b,c)} : (a,b,c) \in \R^{3}\}, \end{displaymath}
where 
\begin{displaymath} \ell_{(a,b,c)} = \{(as + b,s,\tfrac{b}{2}s + c) : s \in \R\}. \end{displaymath}
The subset $\mathcal{L}'(\He)$ consists of all elements of $\mathcal{L}(\He)$, except for those contained in some translate of the plane $\W_{0} := \{(x,0,t) : x,t \in \R\}$. By definition, every set $\mathcal{L} \subset \mathcal{L}'(\He)$ can be written as
\begin{displaymath} \mathcal{L} = \ell(P) := \{\ell_{(a,b,c)} : (a,b,c) \in P\} \end{displaymath}
for some set $P \subset \R^{3}$. This identification of $\mathcal{L}'(\He)$ with $\R^{3}$ allows us to transport notions like "Borel set" and "dimension" from $\R^{3}$ to corresponding notions for subsets of $\mathcal{L}'(\He)$: 

\begin{definition} Let $\mathcal{L} = \ell(P) \subset \mathcal{L}'(\He)$. We say that $\mathcal{L}$ is a Borel set if $P \subset \R^{3}$ is a Borel set. We define $\Hd \mathcal{L} := \Hd P$, where "$\Hd P$" refers to the Euclidean Hausdorff dimension of $P \subset \R^{3}$.  \end{definition}

Now we can state our main result about unions of horizontal lines:

\begin{thm}\label{main} Let $\mathcal{L} \subset \mathcal{L}'(\He)$. Then,
\begin{displaymath} \Hd (\cup \mathcal{L}) = \min\{\Hd \mathcal{L} + 1,3\}. \end{displaymath}
Here "$\Hd (\cup \mathcal{L})$" is the Euclidean Hausdorff dimension of the union $\cup \mathcal{L} := \bigcup_{\ell \in \mathcal{L}} \ell$. \end{thm}

The following corollary for horizontal lines is equivalent to Theorem \ref{mainSL}:

\begin{cor}\label{mainCor} Let $\mathcal{L} \subset \mathcal{L}(\He)$ with $\mathcal{H}^{2}(\mathrm{dir}(\mathcal{L})) > 0$. Then,
\begin{displaymath} \Hd (\cup \mathcal{L}) = 3. \end{displaymath}
\end{cor}

\begin{remark}\label{rem1} Theorem \ref{main} and Corollary \ref{mainCor} continue to hold if full lines are replaced by line segments of positive length, see the discussion below \eqref{form41}. Thus, if $\mathcal{L} \subset \mathcal{L}'(\He)$, and every line $\ell \in \mathcal{L}$ contains a segment $I(\ell) \subset \ell$ of positive length, then
\begin{equation}\label{form47} \Hd \Big( \bigcup_{\ell \in \mathcal{L}} I(\ell) \Big) =\min\{\Hd \mathcal{L} + 1,3\}. \end{equation}
\end{remark}

%\begin{remark} Note that
%\begin{displaymath} \mathrm{dir}(\mathcal{L}(\He)) = S^{2} \, \setminus \, \{(0,0,1),(0,0,-1)\} \end{displaymath}
%So, in particular $\mathcal{H}^{2}(\mathcal{L}(\He)) = \mathcal{H}^{2}(S^{2})$.
%\end{remark}

\subsection{Nikodym sets associated with horizontal lines}

Theorem \ref{main} easily yields information about the dimension of \emph{Nikodym sets} associated with horizontal lines. A set $N \subset \R^{3}$ is called an \emph{$\mathcal{L}(\He)$-Nikodym set} if for every $p \in \R^{3}$ (or more generally every $p \in \R^{3}$ in a measurable set of positive measure $\Omega \subset \R^{3}$) there exists a line $\ell_{p} \in \mathcal{L}(\He)$ containing $p$ such that $N$ contains a line segment $I_{p} \subset \ell_{p}$ of positive length.

\begin{cor}\label{cor1} Every $\mathcal{L}(\He)$-set $N \subset \R^{3}$ has $\Hd N = 3$. \end{cor}

It is well-known that bounds for Kakeya sets yield bounds for Nikodym sets: we only repeat the standard details below for the reader's convenience. For a similar argument in the case of classical Kakeya and Nikodym sets, see \cite[Section 11.3]{MR3617376}.

\begin{proof}[Proof of Corollary \ref{cor1}] We may assume without loss of generality that all the lines $\ell_{p} \in \mathcal{L}(\He)$ appearing in the definition of "$N$" lie in $\mathcal{L}'(\He)$. Namely, if this is true for a positive measure subset of the points $p \in \Omega$, we simply replace $\Omega$ by that subset. If this fails for Lebesgue almost every point $p \in \Omega$, then we apply a rotation $R$ of, say, $10^{\circ}$ around the $t$-axis to the objects $\Omega$, $N$, and the lines $\ell_{p}$, $p \in \Omega$. Rotations around the $t$-axis preserve $\mathcal{L}(\He)$, and the measure and dimension of $\Omega$ and $N$. After this procedure, we moreover have $\ell_{p} \in \mathcal{L}'(\He)$ for a.e. $p \in R(\Omega)$. 

Using Fubini's theorem, start by picking $y_{0} \in \R$ such that $\mathcal{H}^{2}(\Omega \cap \mathbb{W}_{y_{0}}) > 0$. Here $\mathbb{W}_{y} = \{(x,y,t) : x,t \in \R^{3}\}$ for $y \in \R$. By assumption, for every $p = (x,y_{0},t) \in \Omega \cap \mathbb{W}_{y_{0}}$, there exists a line 
\begin{displaymath} \ell_{p} := \ell_{(a(p),b(p),c(p))} \in \mathcal{L}'(\He) \end{displaymath}
containing $p$ such that $N$ contains a line segment $I_{p} \subset \ell_{p}$ of positive length. 

Now, note that the map $(a,b,c) \mapsto \Psi(a,b,c) = (ay_{0} + b,y_{0},\tfrac{b}{2}y_{0} + c)$ is Lipschitz, and 
\begin{displaymath} \Omega \cap \W_{y_{0}} \subset \Psi(\{(a(p),b(p),c(p)) : p \in \Omega \cap \W_{y_{0}}\}). \end{displaymath}
(This is because the lines $\ell_{p}$ contain the points $p \in \Omega \cap \W_{y_{0}}$.) Therefore,
\begin{displaymath} \Hd \{(a(p),b(p),c(p)) : p \in \Omega \cap \W_{y_{0}}\} \geq \Hd (\Omega \cap \W_{y_{0}}) = 2. \end{displaymath}
In particular, the set of lines $\mathcal{L} := \{\ell_{p} : p \in \Omega\} \subset \mathcal{L}'(\He)$ has $\Hd \mathcal{L} \geq 2$ by definition. Therefore, it follows from Theorem \ref{main}, or to be precise \eqref{form47}, that 
\begin{displaymath} \Hd N \geq \Hd \Big(\bigcup_{p \in \Omega} I_{p} \Big) = 3. \end{displaymath}
This completes the proof. \end{proof}

\begin{remark} Nikodym set for "restricted" families of lines were earlier considered by Kim \cite{MR2956810}. Corollary \ref{cor1} answers (a special case of) a question raised on \cite[p. 478]{MR2956810}. We elaborate on this a little further. The paper \cite{MR2956810} considered general families of $2$-planes $p \mapsto \Pi_{\mathfrak{a}}(p) \subset \R^{3}$, where $p \mapsto \mathfrak{a}(p)$ is a non-vanishing measurable vector field, and
\begin{displaymath} p \in \Pi_{\mathfrak{a}}(p) \quad \text{and} \quad \spa (\mathfrak{a}(p)) = \Pi_{\mathfrak{a}}(p)^{\perp}. \end{displaymath}
One can associate Nikodym sets $N \subset \R^{3}$ to such a plane family, as follows: for every $p \in \R^{3}$, the requirement is that there exists a line $\ell_{p} \subset \R^{3}$ satisfying
\begin{displaymath} p \in \ell_{p} \subset \Pi_{\mathfrak{a}}(p), \end{displaymath}
and a non-trivial segment $I_{p} \subset N \cap \ell_{p}$. How small can such a Nikodym set $N \subset \R^{3}$ be? In \cite{MR2956810}, Kim approached the question via maximal function estimates, and his results depend on the properties of the vector field $\mathfrak{a}$. Kim considered vector fields $\mathfrak{a}$ of the form
\begin{displaymath} \mathfrak{a}(p) = (a_{11}p_{1} + a_{21}p_{2},a_{12}p_{1} + a_{22}p_{2},-1), \qquad p = (p_{1},p_{2},p_{3}) \in \R^{3}, \end{displaymath}
and defined the "discriminant" $D_{\mathfrak{a}} = (a_{12} + a_{21})^{2} - 4a_{11}a_{22}$. In \cite[Corollary 1, p. 478]{MR2956810}, it was shown that the dimension of $N$ equals $3$ is if $D_{\mathfrak{a}} \neq 0$. Right after the corollary, the question is raised, what happens in the situation $D_{\mathfrak{a}} = 0$.

Now, recall the definition of horizontal lines $\mathcal{L}(\He)$: these were the lines contained in the planes $\Pi_{p} = p \ast \Pi_{0}$, and passing through $p$. The planes $\Pi_{p}$ fit in the framework of \cite{MR2956810}, choosing $\mathfrak{a}(p) = (-p_{2}/2,p_{1}/2,-1)$, or $(a_{11},a_{12},a_{21},a_{22}) = (0,\tfrac{1}{2},-\tfrac{1}{2},0)$. In particular, $D_{\mathfrak{a}} = 0$. Also, the $\mathcal{L}(\He)$-Nikodym sets defined above Corollary \ref{cor1} are the same as the Nikodym sets of \cite{MR2956810} associated with the planes $\Pi_{p} = p \ast \Pi_{0}$. Thus, Corollary \ref{cor1} covers the special case $(a_{11},a_{12},a_{21},a_{22}) = (0,\tfrac{1}{2},-\tfrac{1}{2},0)$ of the problem raised on \cite[p. 478]{MR2956810}.

\end{remark}

\subsection{Ingredients of the proof} The proof of Theorem \ref{main} is based on two ingredients. The first one is a \emph{point-line duality} between horizontal lines and \emph{conical lines} in $\R^{3}$, namely translates of lines contained in the light cone $\{(x,y,t) : t^{2} = x^{2} + y^{2}\}$. This duality was formalised in our paper \cite{2022arXiv221000458F}, although it was already implicit in the work \cite{MR4439466} of Liu. Using this point-line duality, Kakeya-type problems for horizontal lines can be transformed into projection problems in $\R^{3}$. These projection problems concern "restricted" families of projections to planes in $\R^{3}$. Sharp results for such families were recently established by Gan, Guo, Guth, Harris, Maldague, and Wang \cite{2022arXiv220713844G}. This is the second key component in the proof of Theorem \ref{main}.

\section{Proofs concerning $SL(2)$ lines} 

In this section we formalise the connection between $SL(2)$ lines and horizontal lines. We also deduce our main result, Theorem \ref{mainSL}, from Corollary \ref{mainCor}.

Recall the $SL(2)$ lines from Definition \ref{def:SL}. We write $\mathcal{L}_{SL(2)}'$ for all the lines in $\mathcal{L}_{SL(2)}$, except for the $x$-axis, or lines of the form $L_{\alpha,\beta,\gamma,\delta}$ with $\delta = 0$. The difference between $\mathcal{L}_{SL(2)}$ and $\mathcal{L}_{SL(2)}'$ is the same as the difference between $\mathcal{L}(\He)$ and $\mathcal{L}'(\He)$. Consider the map
\begin{displaymath} \Xi(x,y,t) := (x,y,t/2). \end{displaymath}
We claim that $\Xi$ maps the $SL(2)$ lines to horizontal lines. More precisely:
\begin{proposition}\label{prop1} If $L_{\alpha,\beta,\gamma,\delta} \in \mathcal{L}_{SL(2)}'$ with $\delta \neq 0$ and $\alpha \delta - \beta \gamma = 1$, then
\begin{equation}\label{form45} \Xi(L_{\alpha,\beta,\gamma,\delta}) = \ell_{(a,b,c)} \in \mathcal{L}'(\He), \end{equation}
where 
\begin{displaymath} \begin{cases} a = \gamma/\delta, \\ b = 1/\delta,\\ c = -\beta/(2\delta). \end{cases} \end{displaymath} 
  \end{proposition}

\begin{proof} Fix $\alpha,\beta,\gamma,\delta$ with $\delta \neq 0$ and $\alpha \delta - \beta \gamma = 1$. Write $L_{\alpha,\beta,\gamma,\delta}(s) = (\alpha,\beta,0) + (s\gamma,s\delta,s)$. It is a straightforward computation to check that
\begin{displaymath} \Xi(L_{\alpha,\beta,\gamma,\delta}(s)) = \ell_{(a,b,c)}(\beta + s\delta), \qquad s \in \R. \end{displaymath}
Since $\delta \neq 0$ by assumption, this completes the proof. \end{proof}

We are then prepared to prove Theorem \ref{mainSL}.

\begin{proof}[Proof of Theorem \ref{mainSL}] We may assume that $\mathcal{L} \subset \mathcal{L}_{SL(2)}'$, since the directions of the lines in $\mathcal{L}_{SL(2)} \, \setminus \, \mathcal{L}_{SL(2)}'$ are contained in the $\mathcal{H}^{2}$ null set $S^{2} \cap \{(x,0,t) : x,t \in \R\}$. Similarly, we may assume that $\mathcal{L}$ contains no lines in the $xy$-plane; thus every $L \in \mathcal{L}$ has the form $L = L_{\alpha,\beta,\gamma,\delta}$ for some $\alpha,\beta,\gamma,\delta$ with $\delta \neq 0$ and $\alpha \delta - \beta \gamma = 1$.

Since $\mathcal{L} \subset \mathcal{L}_{SL(2)}'$, we infer from Proposition \ref{prop1} that $\Xi(\mathcal{L}) := \{\Xi(\ell) : \ell \in \mathcal{L}\} \subset \mathcal{L}'(\He)$. We claim that
\begin{equation}\label{form40} \mathcal{H}^{2}(\mathrm{dir}(\Xi(\mathcal{L}))) > 0. \end{equation}
According to Corollary \ref{mainCor}, this will imply that 
\begin{displaymath} \Hd (\cup \mathcal{L}) = \Hd \Xi(\cup \mathcal{L}) = \Hd (\cup \Xi(\mathcal{L})) = 3, \end{displaymath}
and complete the proof.

To verify \eqref{form40}, fix $L = L_{\alpha,\beta,\gamma,\delta} \in \mathcal{L}$. Then, by \eqref{form44}, we have $\Xi(L) = \ell_{(a,b,c)}$ with 
\begin{equation}\label{form46} \begin{cases} a = \gamma/\delta, \\ b = 1/\delta. \end{cases} \end{equation} 
Since $\mathcal{H}^{2}(\mathrm{dir}(\mathcal{L})) > 0$, and the direction of $L_{\alpha,\beta,\gamma,\delta} = (\alpha,\beta,0) + \spa(\gamma,\delta,1)$ is determined by $\gamma$ and $\delta$, we know that
\begin{displaymath} \mathcal{H}^{2}(\{(\gamma,\delta) \in \R^{2} : L_{\alpha,\beta,\gamma,\delta} \in \mathcal{L}\}) > 0. \end{displaymath}
It now follows from \eqref{form46} that also
\begin{displaymath} \mathcal{H}^{2}(\{(a,b) \in \R^{2} : \ell_{(a,b,c)} \in \Xi(\mathcal{L})\}) > 0. \end{displaymath} 
Since the direction of $\ell_{(a,b,c)} = (b,0,c) + \spa(a,1,b/2)$ is determined by $(a,b)$, we may now infer that $\mathcal{H}^{2}(\mathrm{dir}(\Xi(\mathcal{L}))) > 0$, as desired.  \end{proof}

\section{Proofs concerning horizontal lines}

We start by proving Theorem \ref{main}.

\begin{proof}[Proof of Theorem \ref{main}] Without loss of generality, we may assume that $\mathcal{L} = \ell(P)$ is a Borel set of lines, that is, $P \subset \R^{3}$ is a Borel set. For the full details of this reduction, see \cite[Section 3]{MR4439466} or \cite[Theorem 7.9]{MR867284}. The idea is that we can first replace $\cup \mathcal{L}$ by a $G_{\delta}$-set $G \supset \cup \mathcal{L}$ without affecting $\Hd (\cup \mathcal{L})$. Then, it is easy to check that the set of parameters $P' := \{p \in \R^{3} : \ell(p) \subset G\}$ is a Borel set with $P' \supset P$, in particular $\Hd P' \geq \Hd P$. Finally, writing $\mathcal{L}' := \ell(P')$, we have
\begin{displaymath} \Hd (\cup \mathcal{L}) = \Hd G \geq \Hd (\cup \mathcal{L}'). \end{displaymath}
So, if the result is known for Borel sets of lines, it follows for $\mathcal{L}$.

Write $\mathcal{L} := \ell(P)$, where $P \subset \R^{3}$ is Borel. Write also
\begin{displaymath} K_{y} := \{(ay + b,\tfrac{b}{2}y + c) : (a,b,c) \in P\}, \qquad y \in \R, \end{displaymath}
and note that $K_{y}$ is a "slice" of $\cup \mathcal{L}$ with the plane $\mathbb{W}_{y} := \{(x,y,t) : x,t \in \R\}$:
\begin{displaymath} (\cup \mathcal{L}) \cap \mathbb{W}_{y} \cong K_{y}, \end{displaymath}
where "$\cong$" refers to isometry. In order to prove that 
\begin{equation}\label{form42} \Hd (\cup \mathcal{L}) \geq \min\{\Hd \mathcal{L} + 1,3\}, \end{equation}
we now claim that
\begin{equation}\label{form41} \Hd K_{y} = \min\{\Hd P,2\} \quad \text{ for a.e. } y \in \R. \end{equation}
If $\mathcal{L}$ consisted of line segments of positive length, and not full lines, then we would have to modify \eqref{form41} as follows: for every $\epsilon > 0$, there exists an interval $I \subset \R$ of positive length such that $\Hd K_{y} \geq \min\{\Hd P - \epsilon,2\}$ for a.e. $y \in I$. This interval would (be chosen to) consist of points $y \in \R$ with the property that the plane $\mathbb{W}_{y}$ intersects a family of segments corresponding to a $(\Hd P - \epsilon)$-dimensional Borel subset $P' \subset P$. We refer the reader to \cite[Section 3]{MR4439466} for a very similar argument.

Clearly \eqref{form42} follows from \eqref{form41} by the "Fubini inequality" for Hausdorff measures (hence dimension), see \cite[Theorem 5.8]{MR867284} or \cite[Theorem 7.7]{zbMATH01249699}. To prove \eqref{form41}, we define
\begin{displaymath} v(y) := (y,1,0) \quad \text{and} \quad w(y) := (0,y/2,1), \qquad y \in \R. \end{displaymath} 
Then, we note that for $y \in \R$ fixed, $K_{y}$ can be expressed as
\begin{displaymath} K_{y} = \{(\langle p, v(y) \rangle,\langle p, w(y) \rangle ) : p \in P\}, \end{displaymath} 
where "$\langle \cdot,\cdot \rangle$" is the Euclidean dot product. This is a "projection" of $P$ to the plane 
\begin{displaymath} V_{y} := \spa(\{v(y),w(y)\}), \end{displaymath}
but not an orthogonal projection, since $\{v(y),w(y)\}$ is not an orthonormal basis of $V_{y}$. Regardless, 
\begin{equation}\label{form43} \Hd K_{y} = \Hd \pi_{V_{y}}(P), \qquad y \in \R, \end{equation}
because $K_{y}$ is an invertible linear image of $\pi_{V_{y}}(P)$. Let us check this carefully. First, $v(y)$ and $w(y)$ are linearly independent, so we may write (for $y \in \R$ fixed)
\begin{displaymath} \begin{cases} e_{1} := \frac{v(y)}{|v(y)|}, \\ e_{2} := \alpha e_{1} + \beta w(y), \end{cases} \end{displaymath}
where $\{e_{1},e_{2}\}$ is an orthonormal basis of $V_{y}$. The vectors $e_{1},e_{2}$ and the coefficients $\alpha \in \R$ and $\beta \in \R \, \setminus \, \{0\}$ depend on $y$, but we suppress this from the notation.

With this notation, we define the invertible linear map $M_{y} \colon \R^{2} \to \R^{2}$,
\begin{displaymath} M_{y}(z_{1},z_{2}) := (|v(y)|z_{1},\tfrac{1}{\beta}z_{2} - \tfrac{\alpha}{\beta} z_{1}), \qquad (z_{1},z_{2}) \in \R^{2}. \end{displaymath}
Then, one may calculate that
\begin{displaymath} M_{y}(\langle p,e_{1} \rangle,\langle p,e_{2} \rangle) = (\langle p,v(y) \rangle,\langle p,w(y) \rangle), \qquad p \in \R^{3}, \, y \in \R. \end{displaymath} 
The left hand side is the $M_{y}$-image of the orthogonal projection $\pi_{V_{y}}(p)$. Therefore, $K_{y}$ is indeed the $M_{y}$-image of $\pi_{V_{y}}(P)$, hence \eqref{form43} holds.

To complete the proof, we claim that
\begin{equation}\label{form44} \Hd \pi_{V_{y}}(P) = \min\{\Hd P,2\} \quad \text{for a.e. } y \in \R. \end{equation}
The idea is that $\{\pi_{V_{y}}\}_{y \in \R}$ is a $1$-parameter family of orthogonal projections to planes in $\R^{3}$ which satisfies the hypotheses of \cite[Corollary 1]{2022arXiv220713844G}.

Which planes are the planes $V_{y}$? Note that
\begin{displaymath} v(y) \times w(y) = (1,-y,y^{2}/2) =: e_{y}. \end{displaymath}
Thus, $V_{y} = e_{y}^{\perp}$. Moreover, the lines $\ell_{y} := \spa(e_{y})$ are all contained in a $45^{\circ}$ rotated copy of the light cone
\begin{displaymath} \mathcal{C} := \{(x,y,t) \in \R^{3} : t^{2} = x^{2} + y^{2}\}, \end{displaymath}
see \cite[Section 2.2]{2022arXiv221000458F} for the details. This implies that the projections $\{\pi_{V_{y}}\}_{y \in \R}$ satisfy the curvature condition \cite[(1)]{2022arXiv220713844G}. In fact, up to the rotation by $45^{\circ}$, this family of projections is precisely the "model example" mentioned just below \cite[(1)]{2022arXiv220713844G}. Therefore, \eqref{form44} follows from \cite[Corollary 1]{2022arXiv220713844G}, and the proof is complete. \end{proof}

We conclude the paper by proving Corollary \ref{mainCor}.

\begin{proof}[Proof of Corollary \ref{mainCor}] First, note that $\mathcal{H}^{2}(\mathrm{dir}(\mathcal{L} \cap \mathcal{L}'(\He))) > 0$. This is because $\mathrm{dir}(\mathcal{L}'(\He))$ contains all the directions on $S^{2}$, except for those contained in the null set $\{(x,0,t) : x,t \in \R\}$. Therefore, we may assume that $\mathcal{L} \subset \mathcal{L}'(\He)$.

Write $\mathcal{L} = \ell(P)$, where $P \subset \R^{3}$. Recall that
\begin{align*} \mathcal{L} = \ell(P) & = \{(as + b,s,\tfrac{b}{2}s + c) : s \in \R, \, (a,b,c) \in P\}\\
& = \{(b,0,c) + \spa(a,1,\tfrac{b}{2}) : (a,b,c) \in P\}. \end{align*}
Since $\mathcal{H}^{2}(\mathrm{dir}(\mathcal{L})) > 0$ by assumption, we see that
\begin{displaymath} \mathcal{H}^{2}(\{(a,\tfrac{b}{2}) : (a,b,c) \in P\}) > 0, \end{displaymath}
and consequently $\Hd P \geq 2$. The claim now follows from Theorem \ref{main}. \end{proof}

\bibliographystyle{plain}
\bibliography{references}
\end{document}